\documentclass{amsart}
\usepackage[a4paper]{geometry}
\usepackage{ucs}
\usepackage[utf8x]{inputenc}
\usepackage{amsthm}
\usepackage{url}
\usepackage{comment}
\makeatletter

\def\part{\@startsection{part}{0}%
  \z@{\linespacing\@plus\linespacing}{.5\linespacing}%
  {\normalfont\bfseries\centering}}
\makeatother
\newtheorem{thm}{Theorem}[section]
\newtheorem{lem}[thm]{Lemma}
\newtheorem{prop}[thm]{Proposition}
\newtheorem{cor}[thm]{Corollary}

\theoremstyle{remark}

\newcommand{\R}{\mathbb{R}}

\usepackage{hyperref}
\usepackage{xcolor}
\hypersetup{
  colorlinks   = true, 
  urlcolor     = blue, 
  linkcolor    = blue, 
  citecolor   = red 
}

\usepackage{setspace}
\usepackage{amsmath}
\usepackage{amsfonts}
\usepackage{amssymb}

\title[Geometric quantities related to transversality, curvature and visibility]{On some integral-geometric quantities related to transversality, curvature and visibility}
\author{Silouanos Brazitikos, Anthony Carbery and Finlay McIntyre}
\address{Silouanos Brazitikos, Department of Mathematics \& Applied Mathematics, University of Crete,
Voutes Campus, 70013 Heraklion, Greece.}
\email{silouanb@uoc.gr}
\address{Anthony Carbery and Finlay McIntyre, 
School of Mathematics and Maxwell Institute for Mathematical Sciences, 
University of Edinburgh,
James Clerk Maxwell Building, 
Peter Guthrie Tait Road,
King's Buildings, 
Mayfield Road, 
Edinburgh, EH9 3FD, 
Scotland.}
\email{A.Carbery@ed.ac.uk, finlay.mcintyre@gmail.com}
\date{September 2023, revised April 2024.}
\begin{document}
\setcounter{tocdepth}{1}
\begin{abstract}
We consider some integral-geometric quantities that have recently arisen in harmonic analysis and elsewhere, derive some sharp geometric inequalities relating them, and place them in a wider context.
\end{abstract}
\maketitle
\tableofcontents
\section{Introduction}\label{intro}

We define a {\em generalised $d$-hypersurface $\mathbb{S}$} to be a triple $\mathbb{S} = (S, \sigma, v)$ where $(S, \sigma)$ is a $\sigma$-finite measure space and $v: S \to \mathbb{R}^d$ is a measurable vector field taking values in some Euclidean space $\mathbb{R}^d$. For a $j$-tuple of generalised $d$-hypersurfaces $(\mathbb{S}_1, \dots, \mathbb{S}_j$), with $d \geq j$, we consider the quantities 
$$Q^p_j(\mathbb{S}_1, \dots, \mathbb{S}_j) := \left( \int_{S_j} \cdots \int_{S_1} |v_1(x_1) \wedge \dots \wedge v_j(x_j)|^p {\rm d} \sigma_1(x_1) \dots {\rm d} \sigma_j(x_j)\right)^{1/jp}$$
for $0 < p < \infty$ (with the usual modifcations when $p=0$ or $p= \infty$), where $|v_1 \wedge \dots \wedge v_j|$ is the $j$-dimensional volume of the parallelotope generated by the vectors $v_1, \dots, v_j$.  Evidently these quantities measure the $L^p$ global joint transversality of the vector fields $v_1, \dots , v_j$ over $S_1 \times \dots \times S_j$. Note that $$Q_1^p(\mathbb{S}) = \left(\int_S |v(x)|^p {\rm d} \sigma(x)\right)^{1/p} < \infty$$ if $v \in L^p$, and that, by the Riesz-Hadamard inequality,
if each $v_i \in L^p$, 
$$ Q^p_j(\mathbb{S}_1, \dots, \mathbb{S}_j) \leq
\prod_{i=1}^j Q_1^p(\mathbb{S}_j) = \prod_{i=1}^j \left(\int_{S_i} |v_i(x)|^p {\rm d} \sigma_i(x)\right)^{1/p} < \infty.$$
In the case that the $j$ generalised hypersurfaces coincide, that is, $\mathbb{S}_i = \mathbb{S}$ for some fixed $\mathbb{S}$, we adopt the abbreviated notation
$$ Q_j^p(\mathbb{S}) := Q^p_j(\mathbb{S}, \dots, \mathbb{S}),$$
If we take $S$ to be a classical hypersurface in $\mathbb{R}^d$ with its surface measure, and take $v$ to be the Gauss map, $Q_1^p(\mathbb{S})^p$ is the surface area of $S$; in general the quantities $Q_j^p(\mathbb{S})$ are sensitive to the curvature of $S$.

\medskip
\noindent
These quantities have recently arisen in  a variety of contexts, especially when $j=d$. When $p=1$, they enjoy a geometric significance which we shall also recall. For more on this see Sections~\ref{sec:three} and \ref{sec:five} below. The main purpose of this note is to examine the intermediate quantities $Q^p_j$, and to derive some of the monotonicity and extremal properties of the family of quantities $Q^p_j$. 

\medskip
\noindent
{\bf Main results.} The first main result is that for $p$ fixed, the quantity $Q^p_j$ is dominated by weighted geometric means of the quantities
$Q^p_{j'}$ for $j'< j$. It has the flavour of the Loomis--Whitney inequality and its generalisation to the Finner inequalities, and indeed its proof will rely upon Finner's inequalities. 

\medskip
\noindent
We recall the setting for Finner's inequalities \cite{Finner}. We have subsets $A_i \subseteq \{1, \dots, j\}$ and positive numbers $\alpha_i$ for $1 \leq i \leq m$ such that for all $1 \leq l \leq j$, $\sum_{i=1}^m \alpha_i \chi_{A_i}(l) = 1$, then 
$\{(A_i, \alpha_i)\}_{i=1}^m$ is called a {\em uniform cover} of $\{1, \dots , j\}$. For $A \subseteq \{1, \dots , j\}$, let $\Pi_A (\mathbb{S}_1, \dots , \mathbb{S}_j) = (\mathbb{S}_n)_{n \in A}$ be projection. 

\begin{thm}\label{thm:general}
Let $\mathbb{S}_1, \dots ,\mathbb{S}_j$ be generalised $d$-hypersurfaces with $d \geq j$ and let $0 < p < \infty$. Then
$$ Q^p_j(\mathbb{S}_1, \dots ,\mathbb{S}_j) \leq
\prod_{i=1}^j Q^p_{j-1}( \mathbb{S}_1, \dots ,
\widehat{\mathbb{S}_i}, \dots ,
\mathbb{S}_j)^{1/j}.$$
More generally, suppose that $(A_i, \alpha_i)_{i=1}^m$ is a uniform cover of $\{1, \dots , j\}$ and that $|A_i| = k_i$.  Then
$$Q^p_j(\mathbb{S}_1, \dots ,\mathbb{S}_j) \leq
\prod_{i=1}^m Q^p_{k_i}(\Pi_{A_i} (\mathbb{S}_1, \dots , \mathbb{S}_j))^{\alpha_i k_i/j}.$$
In particular, on the diagonal $\mathbb{S}_1 = \dots = \mathbb{S}_j$, the quantities $Q^p_j(\mathbb{S}, \dots, \mathbb{S})$ form a decreasing sequence, i.e. for $1 \leq j \leq d-1$ we have
\begin{equation}\label{const1}
Q^p_{j+1}(\mathbb{S}) \leq Q^p_{j}(\mathbb{S}).
\end{equation}
\end{thm}

\medskip
\noindent
The implicit constants $1$ in the first two statements are sharp, see the remarks after the proof. The proof will also show that in the diagonal case we have strict inequality in \eqref{const1} (except in trivial cases). There is an analogue of Theorem~\ref{thm:general} corresponding to the cases $p=0$ and $p = \infty$,  which may be derived either by a limiting argument, or by adapting the proof presented below in a routine manner. In the case $p=\infty$, however, the constant $1$ is optimal even in the diagonal case, and there it is realised by unit spheres with the gauss map, since the $\sup$-norm of each $x_1 \wedge \dots \wedge x_j$ as the $x_i$ range over a unit sphere is exactly $1$. On the other hand, it turns out that when $p=1$ or
$2$ we can do better, and we can give the sharp inequality corresponding to \eqref{const1}, demonstrating that it is extremised by spheres. This is the content of the second main result. We reserve the notation $ Q_j^{p}(\mathbb{S}^{d-1})$ for $Q_j^{p}((\mathbb{S}^{d-1}, \sigma, \iota))$ where $\sigma$ is the standard surface measure on $\mathbb{S}^{d-1}$, normalised to coincide with induced Lebesgue measure, and $\iota$ is the inclusion map.

\begin{thm}\label{main_baby}\label{main} Let $\mathbb{S}$ be a generalised $d$-hypersurface as above, and let $p \in \{1,2,\infty\}$. For $1 \leq j \leq d-1$ we have
$$Q_{j+1}^p(\mathbb{S}) \leq  \frac{Q_{j+1}^{p}(\mathbb{S}^{d-1})}{Q_j^{p}(\mathbb{S}^{d-1})}\; {Q_j^{p}(\mathbb{S})}.$$
\end{thm}

We remark that explicit values for the quantities $Q_j^{p}(\mathbb{S}^{d-1})$ can be found in the literature, see for example \cite [Theorem~2]{miles_1971}. For a more accessible and recent treatment, the argument in \cite [Remark 8.15]{bennett-tao} can be readily adapted to give a simple formula for $Q_j^{p}(\mathbb{S}^{d-1})$ for arbitrary $j$ and $p$. We leave the details to the interested reader. We observe in passing that 
$$Q_j^{1}(\mathbb{S}^{d-1}) = \omega_{d-1} \left(\frac{\omega_d d!}{\omega_{d-j} (d-j)!}\right)^{1/j}$$where $\omega_j$ denotes the volume of the unit ball in $\mathbb{R}^j$, a result that we shall recover later.
Note that the case $p=\infty$ of Theorem~\ref{main} actually has no content beyond that implied by Theorem~\ref{thm:general}. See Section~\ref{sec:genp} for some remarks on what happens when $p \neq 1,2$.

\medskip
\noindent
{\bf Log-concavity.} We also have the following log-concavity property of the quantities $\left(\frac{Q_j^{p}(\mathbb{S})}{Q_j^{p}(\mathbb{S}^{d-1})}\right)^j$, which in fact underpins Theorem~\ref{main}: 

\begin{thm}\label{logcvx}
For a generalised $d$-hypersurface $\mathbb{S}$, $2 \leq j \leq d-1$ and $p \in \{1,2\}$ we have
$$ \left(\frac{Q_{j-1}^{p}(\mathbb{S})}{Q_{j-1}^{p}(\mathbb{S}^{d-1})}\right)^{j-1} \left(\frac{Q_{j+1}^{p}(\mathbb{S})}{Q_{j+1}^{p}(\mathbb{S}^{d-1})}\right)^{j+1} \leq \left(\frac{Q_j^{p}(\mathbb{S})}{Q_j^{p}(\mathbb{S}^{d-1})}\right)^{2j}.$$

\end{thm}
The first statement of Theorem~\ref{main} follows in a routine way from Theorem~\ref{logcvx} by using the following elementary lemma:

\begin{lem}\label{elem}
Suppose that $a_j > 0$ for $1 \leq j \leq d$, that 
\begin{equation}\label{base}
a_2 \leq a_1^2
\end{equation}
and that
\begin{equation}\label{general}
a_{j-1} a_{j+1} \leq a_j^2 \mbox{ for } 2 \leq j \leq d-1.
\end{equation}
Then 
$$ a_{j+1}^{\frac{1}{j+1}} \leq a_j^{\frac{1}{j}} \mbox{ for } 1 \leq j \leq d-1.$$
Moreover, if we have equality in \eqref{base} and \eqref{general} for all $j$, then $a_j^{1/j}$ is independent of $j$. 

\end{lem}
Indeed, to use this lemma, we simply set $a_j = \left(\frac{Q_j^{p}(\mathbb{S})}{Q_j^{p}(\mathbb{S}^{d-1})}\right)^j$ and note that the condition $a_2 \leq a_1^2$ is the same as
\begin{equation}\label{check}
\frac{Q_2^{p}(\mathbb{S})}{Q_2^{p}(\mathbb{S}^{d-1})} \leq \frac{Q_1^{p}(\mathbb{S})}{Q_1^{p}(\mathbb{S}^{d-1})}
\end{equation}
(which is just the special case $j=1$ of Theorem~\ref{main}). We shall check condition \eqref{base}  at the same time as we prove Theorem~\ref{logcvx} and the remaining assertions of Theorem~\ref{main_baby} in Section~\ref{proof} below.

\medskip
\noindent
{\bf Structure of the paper.} 
In Section~\ref{sec:two} we prove Theorem
~\ref{thm:general} using rather elementary arguments. In Section~\ref{sec:three} we introduce the considerations from convex geometry which we will need for the proof of Theorem~\ref{main} in Section~\ref{proof}. We give some further geometric descriptions of the off-diagonal quantities $Q_j^1(\mathbb{S}_1, \dots, \mathbb{S}_j)$ in Section~\ref{sec:convex}.
Further contextual remarks are made in Section~\ref{sec:five}. Finally, in Section~\ref{appendix:intersections}, we give an elementary argument for a Crofton-like formula from integral geometry which is useful in deriving upper bounds for the quantities $Q_j^1(\mathbb{S})$. 

\medskip
\noindent
Several of the results in this note also appear in the PhD thesis \cite{Finlay} of the third author.

\section{Proof of Theorem~\ref{thm:general}}\label{sec:two}
Suppose that $A_i \subseteq \{1 , \dots, j\}$ for $1 \leq i \leq m$, $|A_i| = k_i$ and that $\sum_{i=1}^m \alpha_i \chi_{A_i}(l) = 1$ for all $1 \leq l \leq j$. We first need a seemingly elementary geometrical inequality:

\begin{lem}\label{lemma:elem}
Let $A_i$ and $\alpha_i$ be as above. Then for any vectors $v_1, \dots , v_j$ in $\mathbb{R}^d$ we have
$$ | v_1 \wedge \dots \wedge v_j|\leq \prod_{i=1}^m \left| \bigwedge_{n \in A_i} v_n \right|^{\alpha_i}$$
with equality when the $v_n$ are mutually orthogonal. %{\color{red} and strict inequality if...?}
\end{lem}
\begin{proof}
We identify the quotient
$$\frac{\prod_{i=1}^m \left| \bigwedge_{n \in A_i} v_n \right|^{\alpha_i}}{| v_1 \wedge \dots \wedge v_j|}$$
as an upper bound for an affine-invariant Brascamp--Lieb inequality of Finner type as in Theorem~1.3 of \cite{ABBC}. On the other hand, Valdimarsson has shown in \cite{Valdimarsson} that the smallest constant in any Brascamp--Lieb inequality of this type is at least $1$. The equality statement is trivial.
\end{proof}
We wish to show that
$$Q^p_j(\mathbb{S}_1, \dots ,\mathbb{S}_j) \leq
\prod_{i=1}^m Q^p_{k_i}(\Pi_{A_i} (\mathbb{S}_1, \dots , \mathbb{S}_j))^{\alpha_i k_i/j}.$$
By the lemma,
$$ Q^p_j(\mathbb{S}_1, \dots ,\mathbb{S}_j)^{jp} =
\int_{S_j} \dots \int_{S_1} |v_1(x_1) \wedge \dots \wedge v_j(x_j)|^p {\rm d} \sigma_1 (x_1) \dots 
{\rm d} \sigma_j (x_j) 
$$
$$ \leq \int_{S_j} \dots \int_{S_1} \prod_{i=1}^m \left| \bigwedge_{n \in A_i} v_n(x_n) \right|^{\alpha_i p} {\rm d} \sigma_1 (x_1) \dots 
{\rm d} \sigma_j (x_j)$$
$$ = \int_{S_j} \dots \int_{S_1} \prod_{i=1}^m F_i((x_n)_{n \in A_i})^{\alpha_i} {\rm d} \sigma_1 (x_1) \dots 
{\rm d} \sigma_j (x_j)$$
where 
$$F_i((x_n)_{n \in A_i}) = \left| \bigwedge_{n \in A_i} v_n(x_n) \right|^{p}.$$
Note that we have equality here if $\{v_1(x_1), \dots , v_j(x_j)\}$ are mutually orthogonal for all $x_i \in S_i$.
By the abstract Finner inequality \cite{Finner}
we have
$$\int_{S_j} \dots \int_{S_1} \prod_{i=1}^m F_i((x_n)_{n \in A_i})^{\alpha_i} {\rm d} \sigma_1 (x_1) \dots 
{\rm d} \sigma_j (x_j) \leq \prod_{i=1}^m \left(\int_{\prod_{n \in A_i} S_n} F_i((x_n)_{n \in A_i}) \prod_{n \in A_i} {\rm d} \sigma_n(x_n)\right)^{\alpha_i}  $$
$$ = \prod_{i=1}^m Q^p_{k_i}((\mathbb{S}_n)_{n \in A_i})^{\alpha_i k_i p}$$
with equality if each $F_i((x_n)_{n \in A_i})$ is a product of functions of the single variable $x_n$ over $n \in A_i$.
Therefore
$$ 
 Q^p_j(\mathbb{S}_1, \dots ,\mathbb{S}_j) \leq 
 \prod_{i=1}^m Q^p_{k_i}((\mathbb{S}_n)_{n \in A_i})^{\frac{\alpha_i k_i}{j}}.$$
 \qed
  
 \medskip
 \noindent
 Let $\mathbb{S}_i$ be the unit cube in the coordinate hyperplane in $\mathbb{R}^{j}$ which is perpendicular to $e_i$,
 together with Lebesgue measure and unit normal $e_i$. In this case 
 $Q^p_j(\mathbb{S}_1, \dots ,\mathbb{S}_j)$ and each $Q^p_{k_i}((\mathbb{S}_n)_{n \in A_i})$ all equal $1$, and so we see that we cannot improve the constant in Theorem~\ref{thm:general} to anything smaller than $1$. On the other hand, in the diagonal setting where each $\mathbb{S}_i = \mathbb{S}$, an examination of the proof shows that (except in certain trivial cases), we shall have strict inequality in the conclusion.

\section{Background and preliminaries from convex geometry}\label{sec:three}
In this section we relate the quantities $Q^p_j(\mathbb{S})$ to the notion of visibility which has arisen in harmonic analysis \cite{Guth} in connection with the multilinear Kakeya problem, and to certain notions which are implicit in the convex geometry literature, see especially \cite{SChneider}. 

\medskip
\noindent
Let $\mathbb{S}= (S, \sigma, v)$ be a generalised $d$-hypersurface, let $1 \leq p < \infty$ and let 
$$ K^p = K^p(\mathbb{S}) := \left\{ y \in \mathbb{R}^d \, : \, \left(\int_S |y \cdot v(x)|^p {\rm d} \sigma(x)\right)^{1/p} \leq 1\right\}.$$
Then $K^p$ is a closed, balanced and convex subset of $\mathbb{R}^d$ which has nonempty interior, and which under certain mild conditions on $\mathbb{S}$ will be compact. We will assume such conditions on the convex body $K^p$ in what follows.

\medskip
\noindent
We define the {\bf $p$-visibility} of $\mathbb{S}$ by 
$${\rm vis}^p(\mathbb{S}) := {\rm vol}(K^p(\mathbb{S}))^{-1/d}.$$
Note that this definition differs from some of the literature (\cite{Guth, Zhang, Z-K}), where, in the case $p=1$, ${\rm vis}^{1}(\mathbb{S})$ is taken to be ${\rm vol}(K^{1}(\mathbb{S}))^{-1}$ rather than ${\rm vol}(K^{1}(\mathbb{S}))^{-1/d}$ .

\subsection{Crude estimates for visibility.} In this subsection we shall be concerned with estimates relating visibility with $Q^p_d(\mathbb{S})$ which are permitted to depend on the dimension $d$ and the index $p$ in some unquantified way. With this in mind we shall use the notations $A \lesssim_d B$, resp. $A\sim_d B$, to mean that the ratio $A/B$ is bounded above, resp. above and below, by quantities depending on $d$. It is well-known that all norms on a finite-dimensional vector space are equivalent, so it suffices to consider $p=1$, and then suppress it. 

\medskip
\noindent
Any compact convex balanced set $K \subseteq \mathbb{R}^d$ with nonempty interior has a (John) ellipsoid $E(K)$ of maximal volume contained in it, and indeed ${\rm vol}(E) \gtrsim_d{\rm vol}(K)$. This $E(K)$ has principal directions $e_1, \dots , e_d$ forming an orthonormal basis of $\mathbb{R}^d$, and semiaxes $l_1, \dots , l_d$. Thus ${\rm vol} (E(K)) \sim_d l_1, \dots l_d$. So, 
$$ {\rm vis}(\mathbb{S}) \sim_d (l_1, \dots l_d)^{-1/d}.$$
Define the {\bf directional mass} of $\mathbb{S}$ in direction $e \in \mathbb{S}^{d-1}$ by
$$ \sigma(e,\mathbb{S}) := \int_S |e \cdot v(x)| {\rm d} \sigma(x).$$
We see that 
$$ \sigma(e_j,\mathbb{S}) = l_j^{-1} \int_S |l_j e_j \cdot v(x)| {\rm d} \sigma(x) \sim l_j^{-1},$$ and so
$${\rm vis}(\mathbb{S}) \sim_d \left(\sigma(e_1,\mathbb{S}) \dots \sigma(e_d,\mathbb{S})\right)^{1/d} \lesssim \sigma(e_1, \mathbb{S}) + \dots + \sigma(e_d, \mathbb{S}) \sim \sigma(S).$$

Now let $u \in \mathbb{S}^{n-1}$. Then $u/ \sigma(u, \mathbb{S}) \in K(\mathbb{S})$. Therefore $K(\mathbb{S})$ contains the convex hull of the vectors $u_1/ \sigma(u_1, \mathbb{S}), \dots , u_d/ \sigma(u_d, \mathbb{S})$, no matter which unit vectors $u_1 \dots, u_d$ we choose. This convex hull has volume 
$$ \sim_d \frac{|u_1 \wedge \dots \wedge u_d|}{\sigma(u_1, \mathbb{S}) \dots \sigma(u_d, \mathbb{S})}.$$

Therefore
$$ {\rm vis}(\mathbb{S}) = {\rm vol}(K(\mathbb{S}))^{-1/d} \lesssim_d \inf_{u_1, \dots , u_d \in \mathbb{S}^{d-1}} \left(\frac{\sigma(u_1, \mathbb{S}) \dots \sigma(u_d, \mathbb{S})}{|u_1 \wedge \dots \wedge u_d|}\right)^{1/d},$$  
and since we already have ${\rm vis}(\mathbb{S}) \sim_d  \left(\sigma(e_1, \mathbb{S}) \dots \sigma(e_d, \mathbb{S})\right)^{1/d}$, we actually have: 
\begin{prop}\label{prop:vislines}
$$ {\rm vis}(\mathbb{S}) \sim_d \inf_{u_1, \dots , u_d \in \mathbb{S}^{d-1}} \left(\frac{\sigma(u_1, \mathbb{S}) \dots \sigma(u_d, \mathbb{S})}{|u_1 \wedge \dots \wedge u_d|}\right)^{1/d},$$
and the infimum is essentially achieved by the orthonormal basis consisting of the principal directions of the John ellipsoid for $K(\mathbb{S})$.
\end{prop}

\medskip
\noindent
We give an off-diagonal version of the orthogonal case of Proposition~\ref{prop:vislines} in Section~\ref{sec:convex} below. 

\subsubsection{Visibility, $k$-planes and $Q_d^1(\mathbb{S})$}
Visibility has a nice interpretation in terms of members of $\mathbb{S}^{d-1}$, or equivalently with lines through $0$. We now give an analogue of Proposition~\ref{prop:vislines} for $k_j$-planes $E_j$ with $k_1 + \dots + k_m =d$. First of all, we need an analogue $\sigma(E, \mathbb{S})$ of $\sigma(e, \mathbb{S})$ when $0 \in E \in \mathcal{G}_{d,k}$ (the Grassmanian manifold of all $k$-planes in $\mathbb{R}^d)$ is a $k$-plane. We make the following definition, following Zhang \cite{Zhang}:

\medskip
\noindent
{\bf Definition.} For $E \in \mathcal{G}_{d,k}$ we define
$$\sigma(E,\mathbb{S}) := \int_S \dots \int_S |E^\perp \wedge v(x_1) \wedge \dots \wedge v(x_k)| {\rm d} \sigma(x_1) \dots  {\rm d} \sigma(x_k).$$
It is perhaps appropriate to comment on why this definition is natural. One perspective on this comes from convex geometry. Let $\pi_V$ denote orthogonal projection onto a subspace $V$ of $\mathbb{R}^d$. Any convex body $L$ has a naturally associated {\em projection body}\footnote{ We adopt the notation $\tilde{\Pi}L$ to denote this classical projection body of a convex body $L$, in contrast to the projection body $\Pi (\mathbb{S})$ of a generalised hypersurface which we shall introduce in Section~\ref{sec:refined}. When $\mathbb{S}$ is the boundary $\partial L$ of a convex body $L$, together with surface measure and the Gauss map, we will have ${\Pi}(\partial L) =2 \tilde{\Pi}(L)$. See Section~\ref{sec:refined} for further discussion.}
$\tilde{\Pi} L$ which is characterised by the property 
$$\sigma(e,\partial L) = 2|\pi_{e^\perp} L|:=2|\pi_{e} \tilde{\Pi}L|$$ for all $e \in \mathbb{S}^{d-1}$, see \cite[Relation (5.80)]{SChneider}. With the definition of $\sigma(E, \mathbb{S})$ that we have given, it is the case that $\sigma(E, \partial L)=c_{d,k}|\pi_{E} \tilde{\Pi}L|$  for all convex bodies $L$ and for all $E \in \mathcal{G}_{d,k}$. This is a combination of \cite[Relation (5.68)]{SChneider} and Corollary \ref{cor:schneider_digested} presented below. For more details see the proof in \cite[Proposition 4.44]{Finlay}. More prosaically, for $k=1$, $|e^\perp \wedge v| = |e \cdot v|$, so this definition manifestly extends that of $\sigma(e, \mathbb{S})$. Note that for $k=d$, with $E = \mathbb{R}^d$, we recover the quantity $Q_d^1(\mathbb{S})$.

\medskip
\noindent
The next proposition extends Proposition~\ref{prop:vislines}:
\begin{prop}\label{prop:visplanes}
Let $k_1 + \dots + k_m = d$. We have
$$ {\rm vis}(\mathbb{S}) \sim_d \inf_{E_j \in \mathcal{G}_{d, k_j}} \left(\frac{\sigma(E_1, \mathbb{S}) \dots \sigma(E_m, \mathbb{S})}{|E_1 \wedge \dots \wedge E_m|}\right)^{1/d},$$
and the infimum is essentially achieved when each $E_j$ is the span of some $k_j$ vectors from the principal directions of the John ellipsoid of $K(\mathbb{S})$. 
\end{prop}
By applying this with $m=1$ we obtain the following characterisation of the quantity $Q_d^1(\mathbb{S})$ in terms of visibility: 
\begin{cor}\label{cor:visaswedge}
We have
$$ {\rm vis}(\mathbb{S}) \sim_d \left(\int_S \dots \int_S |v(x_1) \wedge \dots \wedge v(x_d)| {\rm d} \sigma(x_1) \dots  {\rm d} \sigma(x_d)\right)^{1/d} = \; Q_d^1(\mathbb{S}).$$
\end{cor}
One may prove the $\lesssim_d$ assertion of the corollary directly via elementary considerations from convex geometry, together with some clever definition chasing and an induction\footnote{This argument does not rely upon Proposition~\ref{prop:visplanes}.}; we refer to the simple and elegant argument which can be found in \cite{Zhang}. We will make use of this observation in the proof of the $\lesssim_d$ assertion of Proposition~\ref{prop:visplanes}. 

\medskip
\noindent
{\em Proof of Proposition~\ref{prop:visplanes}.} We begin with the $\lesssim_d$ assertion. For $0 \in E \in \mathcal{G}_{d,k}$ let $V(x) = \pi_E v(x)$ where $\pi_E$ denotes orthogonal projection onto $E$. Then we have that $E^\perp \wedge v(x_1) \wedge \dots \wedge v(x_k) = V(x_1) \wedge \dots \wedge V(x_k)$, and so by the $\lesssim_d$ assertion of the corollary (with $d=k$ applied in the underlying space $E$) we have
$$ \sigma(E, \mathbb{S}) = \int_S \dots \int_S |V(x_1) \wedge \dots \wedge V(x_k)| {\rm d} \sigma(x_1) \dots  {\rm d} \sigma(x_k) \gtrsim_d 
{\rm vol}_k(E \cap K(\mathbb{S}))^{-1}.$$
It is easy to see that ${\rm vol}_k(E \cap K)^{-1} \sim_d {\rm vol}_{d-k}(\pi_{E^\perp} K)/{\rm vol}_{d}(K)$, (with explicit constants it is called the Rogers--Shephard inequality, see for example \cite[Lemma~1.5.6]{GianMilArt}), so we deduce that 
$$ \sigma(E, \mathbb{S}) \gtrsim_d {\rm vol}_{d-k}(\pi_{E^\perp} K)/{\rm vol}_{d}(K).$$

\medskip
\noindent
Now let $0 \in E_j \in \mathcal{G}_{d,k_j}$ with $k_1 + \dots + k_m =d$. We thus have
$$ \frac{\sigma(E_1, \mathbb{S}) \dots \sigma(E_m, \mathbb{S})}{|E_1 \wedge \dots \wedge E_m|} \gtrsim_d \frac{{\rm vol}_{d-k_1}(\pi_{E_1^\perp} K) \dots {\rm vol}_{d-k_m}(\pi_{E_m^\perp} K)}{|E_1 \wedge \dots \wedge E_m| {\rm vol}_{d}(K)^m},$$ 
or
$$ {\rm vol}_{d}(K) \frac{\sigma(E_1, \mathbb{S}) \dots \sigma(E_m, \mathbb{S})}{|E_1 \wedge \dots \wedge E_n|} \gtrsim_d\frac{{\rm vol}_{d-k_1}(\pi_{E_1^\perp} K) \dots {\rm vol}_{d-k_m}(\pi_{E_m^\perp} K)}{|E_1 \wedge \dots \wedge E_m| {\rm vol}_{d}(K)^{m-1}} \gtrsim_d 1$$
since we have the affine-invariant Loomis--Whitney inequality (see for example  \cite[Theorems~1.2 and 1.3]{ABBC})
$$ {\rm vol}_{d}(K) \lesssim_d \left(\frac{{\rm vol}_{d-k_1}(\pi_{E_1^\perp} K) \dots {\rm vol}_{d-k_m}(\pi_{E_m^\perp} K)}{|E_1 \wedge \dots \wedge E_m|} \right)^{1/(m-1)}.$$

Therefore 
$$ {\rm vis}(\mathbb{S}) \lesssim_d \inf_{E_j \in \mathcal{G}_{d,k_j}} \left(\frac{\sigma(E_1, \mathbb{S}) \dots \sigma(E_m, \mathbb{S})}{|E_1 \wedge \dots \wedge E_m|}\right)^{1/d}.$$

We now show that the expression on the right here is essentially mimimised when each $E_j$ is the span of some $k_j$ vectors in $\{e_1, \dots, e_d\}$ (where the $e_j$ are the principal directions of the John ellipsoid of $K$). For this we need a lemma:

\begin{lem}
Suppose that $e_1, \dots , e_k$ are principal directions of the John ellipsoid of $K(\mathbb{S})$, and let $E$ be the span of $e_1, \dots , e_k$. Then
$$ \sigma(E, \mathbb{S}) \lesssim_d \sigma(e_1, \mathbb{S}) \dots \sigma(e_k, \mathbb{S}).$$
\end{lem}
\begin{proof}
Write $v(x) = \sum_{i=1}^d v_i(x) e_i$ where $v_i(x) = v(x) \cdot e_i$. Then $E^\perp \wedge v(x_1) \dots \wedge v(x_k)$ is a sum of terms of the form $ v_{i_1}(x_1) \dots v_{i_k}(x_k)$ where $i_1, \dots, i_k$ is some permutation of $1, \dots , k$. Hence $$\sigma(E,\mathbb{S}) = \int_S \dots \int_S |E^\perp \wedge v(x_1) \wedge \dots \wedge v(x_k)| {\rm d} \sigma(x_1) \dots  {\rm d} \sigma(x_k)$$ 
is dominated by a sum of terms of the form 
$$\int_S \dots \int_S|v_{i_1}(x_1)| \dots |v_{i_k}(x_k)| {\rm d} \sigma(x_1) \dots  {\rm d} \sigma(x_k) = \prod_{j=1}^k\int_S |v_{i_j}(x)| {\rm d} \sigma(x)$$
$$= \prod_{j=1}^k\int_S |v_{j}(x)| {\rm d} \sigma(x) = \prod_{j=1}^k\sigma(e_j, \mathbb{S})$$
\end{proof}

\medskip
\noindent
So with $E_j$ as in the statement of the proposition,  repeated application of this lemma gives
$$ \left(\frac{\sigma(E_1, \mathbb{S}) \dots \sigma(E_m, \mathbb{S})}{|E_1 \wedge \dots \wedge E_m|}\right)^{1/d}  \sim_d\left(\sigma(E_1, \mathbb{S}) \dots \sigma(E_m, \mathbb{S})\right)^{1/d}$$
$$ \lesssim_d \left(\sigma(e_1, \mathbb{S}) \dots \sigma(e_n, \mathbb{S})\right)^{1/d} \sim_d {\rm vis}(\mathbb{S}),$$
by Proposition~\ref{prop:vislines}. \qed

\medskip
\noindent
For an off-diagonal version of the orthogonal case of Proposition~\ref{prop:visplanes} see Section~\ref{sec:convex} below. 

\medskip
\noindent
While this discussion gives a rough geometric feeling for the quantities $Q^p_d(\mathbb{S})$, and it suffices for various purposes in harmonic analysis \cite{Guth, Zhang, Gressman1, Gressman2, Gressman3}, it nevertheless suffers from three defects: firstly, it is limited to the case $j=d$, (see however Section~\ref{sec:convex}), secondly it is insensitive to different values of $p$, and thirdly it is much too crude for the needs of Theorem~\ref{logcvx}. We now begin to address these issues.

\subsection{A more refined identity for $p=1$ -- projection bodies and mixed volumes}\label{sec:refined}
For the rather precise Theorem~\ref{logcvx} we shall need a less crude analysis than what was in the previous subsection. 
We will begin by deconstructing the construction of the convex body $K(\mathbb{S}) =K^1(\mathbb{S})$ corresponding to the case $p=1$.

\medskip
\noindent
For a generalised $d$-hypersurface $\mathbb{S}$ and $y \in \mathbb{R}^d$ let  
$$ h_\mathbb{S}(y) = \int_S |y \cdot v(x)|{\rm d} \sigma(x)$$
so that
$$ K(\mathbb{S}) = \{y \in \mathbb{R}^d \, : \, h_\mathbb{S}(y) \leq 1\}.$$
Note that $h_\mathbb{S}$ is subadditive and positively homogeneous of degree $1$, and is therefore convex. 

\medskip
\noindent
We shall need the notion of the support function of a compact convex set $L \subseteq \mathbb{R}^d$. This is defined as the function $h_L: \mathbb{R}^d \to \mathbb{R}$ given by
$$ h_L(x) = \sup \{x \cdot  y  \, : \, y \in L\}.$$
It is readily seen that $h_L$ is subadditive and positively homogeneous of degree $1$, and thus convex. These properties characterise the class of support functions:

\begin{lem}\label{sublin}\cite[Theorem~1.7.1]{SChneider}
Suppose $h: \R^d \to \R$ is subadditive and positively homogeneous of degree $1$. Then there is a unique compact convex set $L \subseteq \R^d$ whose support function is $h$. 
\end{lem}

\begin{proof}
We simply let $L$ be the intersection of all compact convex sets 
$K$ such that $h_K(x) \geq h(x)$ for all $x \in \mathbb{R}^d$. 
\end{proof}

Thus, for a generalised $d$-hypersurface $\mathbb{S}$, there is a unique compact convex set $\Pi(\mathbb{S})$ such that $h_\mathbb{S} = h_{\Pi(\mathbb{S})}$. This body, $\Pi(\mathbb{S})$, is called the {\bf projection body} of $\mathbb{S}$.

\medskip
\noindent
This is  consistent with the classical notion of the projection body of a compact convex set, up to a factor of $2$. For the reader's convenience we recall here the classical definition of the projection body $\tilde{\Pi} L$ of a convex body $L$. It is defined by its support function. $$h_{\tilde{\Pi} L}(e)=|\pi_{e^{\perp}}L|,$$ for all unit vectors $e$. Note that this is exactly $\frac{1}{2} \sigma(e, \partial L)$, and we immediately obtain $\Pi(\partial L, \sigma, n) = 2 \tilde{\Pi} L$. 

\medskip
\noindent
For example, it is easily checked that 
$$ \Pi(\mathbb{S}^{d-1}) = 2 \omega_{d-1} \mathbb{B}^d;$$
this is because we have $$\int_{\mathbb{S}^{d-1}} |e \cdot x| \,  {\rm d} \sigma(x) = 2 \omega_{d-1}$$
for all unit vectors $e$.

\medskip
\noindent
The two convex bodies $K(\mathbb{S})$ and $\Pi(\mathbb{S})$ are in polarity with each other:
\begin{lem}\label{lem:polarity}
$$ K(\mathbb{S}) = \{ x \, : \, x \cdot y \leq 1 \mbox{ for all } y  \in \Pi(\mathbb{S})\}$$
and 
$$ \Pi(\mathbb{S}) = \{ x \, : \, x \cdot y \leq 1 \mbox{ for all } y  \in K(\mathbb{S})\}.$$
\end{lem}
We leave the easy verification to the reader. Now the volumes of two convex bodies $K, K^\ast$ in polarity scale in a reciprocal relation, ${\rm vol}(K^\ast) \sim_d {\rm vol}(K)^{-1}$; more precisely, there are dimensional constants $c_d$ and $C_d$ so that 
\begin{equation}\label{Santalo} c_d \leq  {\rm vol}(K){\rm vol}(K^\ast) \leq C_d,
\end{equation}
see for example \cite[Theorem~1.5.10, Theorem~8.2.2]{GianMilArt}. It transpires that there is a much closer relationship between the quantities $Q_d^1(\mathbb{S})$ and ${\rm vol}(\Pi(\mathbb{S}))^{1/d}$ than the one relating $Q_d^1(\mathbb{S})$ with ${\rm vis}(\mathbb{S})= {\rm vol}(K(\mathbb{S}))^{-1/d}$ which we derived in the previous subsection. Indeed, we have:

\begin{thm}\label{schneider_digested_baby}
Let $\mathbb{S}$ be a generalised $d$-hypersurface. Then
$$Q_d^1(\mathbb{S})^d = \int_{S} \dots \int_{S} |v(x_1) \wedge \dots \wedge v(x_d)| {\rm d} \sigma(x_1) \dots {\rm d} \sigma(x_d) 
= \frac{d!}{2^d} {\rm vol}(\Pi (\mathbb{S})).$$
\end{thm}
In fact, there is a generalisation of this result for $d$-tuples of generalised $d$-hypersurfaces. This requires the notion of {\bf mixed volumes}.

\subsubsection{Mixed volumes}
Next we introduce a fundamental quantity from classical convex geometry: the {\bf mixed volume}. This quantity can readily be constructed, but for our purpose its properties are more important. The following well-known result describing it axiomatically is discussed in \cite{leinster} and \cite{zhang2020polyhedra}, and is closely related to Hadwiger's theorem \cite[Section~9.1]{MR1608265} and \cite[Theorem~2.2]{MR2057247}.

\begin{thm}\label{mixed_volume_axioms}
Let $\mathcal{K}^d$ denote the class of convex compact sets in $\mathbb{R}^d$. Then there is a unique function
$$V : {\underbrace{\mathcal{K}^d \times \dots \times \mathcal{K}^d}_d} \to \mathbb{R}_+$$ such that
	\begin{enumerate}
		\item \textbf{Volume:} For every $K\in \mathcal{K}^d$, $V(\underbrace{K,\dots,K}_d) = {\rm vol} (K)$
		\item \textbf{Symmetry:} $V$ is symmetric in its arguments
		\medskip
		\item \textbf{Multilinearity:} $V$ is linear in each argument with respect to Minkowski addition and multiplication by non-negative reals. 
	
	\end{enumerate}
\end{thm}
We recall that the Minkowski sum of two sets is given by $K + L :=\{x+y \, : x \in K, y\in L\}$.
We suppress the dependence on $d$ in $V = V_d$, and note that when we write $V(K_1, \dots, K_j)$ we are careful to ensure that the $K_i$ belong in $\mathcal{K}^j$. Furthermore we note that the mixed volume is translation-invariant in each argument.

\medskip
\noindent
The notion of mixed volume may be used to give an effective formula for the quantities $Q_j^1(\mathbb{S}_1, \dots , \mathbb{S}_j)$ for $0 \leq j \leq d$:
\begin{thm}\label{thm:schneider_digested}
Let $\mathbb{S}_1, \dots , \mathbb{S}_j$ be generalised $d$-hypersurfaces with $d \geq j$. Then
$$Q_j^1(\mathbb{S}_1, \dots , \mathbb{S}_j)^j = \int_{S_j} \dots \int_{S_1} |v_1(x_1) \wedge \dots \wedge v_j(x_j)| {\rm d} \sigma_1(x_1) \dots {\rm d} \sigma_j(x_j) $$
$$ = \frac{d!}{2^j (d-j)! \omega_{d-j}} V( \Pi \mathbb{S}_1, \dots , \Pi \mathbb{S}_j, \mathbb{B}^d [d-j]).$$
\end{thm}
The notation $\mathbb{B}^d [d-j]$ signifies $(d-j)$ copies of $\mathbb{B}^d$. 

\medskip
\noindent
This result is a mild generalisation of the special case of \cite[Theorem~5.3.2]{SChneider} in which the zonoid is taken to be the projection body and each $K_i$ is the unit ball. The proof there is presented using the classical definition of the projection body, but it goes through without significant alteration using our formulation of the projection body, leading to the factors $2^j$ occurring on the right-hand side.

\medskip
\noindent
From this formula, together with the fact that $\Pi(\mathbb{S}^{d-1}) = 2 \omega_{d-1} \mathbb{B}^d$ and the properties of mixed volumes, we may read off the value of $Q^1_j(\mathbb{S}^{d-1})^j$ as
$$ Q^1_j(\mathbb{S}^{d-1})^j =  \frac{d! 2^j \omega_{d-1}^j \omega_d}{2^j (d-j)! \omega_{d-j}} = \frac{d! \omega_{d-1}^j \omega_d}{(d-j)! \omega_{d-j}}$$
as asserted in the introduction. 

\medskip
\noindent
By taking $j=d$ in Theorem~\ref{thm:schneider_digested} we obtain the generalisation of Theorem~\ref{schneider_digested_baby} for $d$-tuples of generalised $d$-hypersurfaces:
\begin{cor}\label{cor:schneider_digested}
Let $\mathbb{S}_1, \dots , \mathbb{S}_d$ be generalised $d$-hypersurfaces. Then
$$Q_d^1(\mathbb{S}_1, \dots \mathbb{S}_d)^d = \int_{S_d} \dots \int_{S_1} |v_1(x_1) \wedge \dots \wedge v_d(x_d)| {\rm d} \sigma_1(x_1) \dots {\rm d} \sigma_d(x_d) $$
$$ = \frac{d!}{2^d} V( \Pi (\mathbb{S}_1), \dots , \Pi (\mathbb{S}_d)).$$
\end{cor}

\subsection{Mixed discriminants -- the case $p=2$} For more precise relations when $p=2$, we do not directly consider the convex bodies $K^2(\mathbb{S})$ and their polar bodies, together with mixed volumes.
As in \cite{Brazitikos--McIntyre}, the case $p=2$ is related to matters considering determinants. In particular, the notion that we need here is the so-called mixed discriminant. Just as our quantities $Q_j^1$ can be expressed in terms of intrinsic volumes, the quantities $Q_j^2$ can be expressed in terms of mixed discriminants.

\medskip
\noindent
Recall that if $T_1,\ldots ,T_m$
are positive semi-definite $d\times d$ matrices, then the determinant of $a_1T_1+\cdots +a_mT_m$ can be expanded as a homogeneous
polynomial of degree $d$ in $a_1,\ldots ,a_m\geq 0$; one has
\begin{equation}\det\left(a_1T_1+\cdots +a_mT_m\right) =\sum_{1\leq i_1,\ldots ,i_d\leq m}D(T_{i_1},\ldots ,T_{i_d})a_{i_1}\cdots a_{i_d},\end{equation}
where the coefficient $D(T_{i_1},\ldots ,T_{i_d})$ depends only on $i_1,\ldots ,i_d$ and is invariant under permutations of the $i_j$'s.
This coefficient is the mixed discriminant of $T_{i_1},\ldots ,T_{i_d}$.

\medskip
\noindent
We shall need a number of properties of mixed discriminants:

\begin{lem}\cite[Lemma~2]{Bapat}\label{lem:Bapat} If $S,T,T_i, T_i^{\prime }$ are positive
semi-definite $d\times d$ matrices, then:
\begin{enumerate}
\item[{\rm (i)}] $D(T_1,\ldots ,T_d)\geq 0$.
\item[{\rm (ii)}] $D(T,T,\ldots ,T)=\det (T)$. In particular, $D(I_d,\ldots ,I_d)=1$.
\item[{\rm (iii)}] $dD(T,I_d,\ldots ,I_d)={\rm tr}(T)$.
\item[{\rm (iv)}] $D(a T_1+b T_1^{\prime },T_2,\ldots ,T_d)=a D(T_1,T_2,\ldots ,T_d)+b D(T_1^{\prime },T_2,\ldots ,T_d)$
for all $a,b\geq 0$.
\item[{\rm (v)}] $D(T_1S,T_2S,\ldots ,T_dS)=|\det (S)|D(T_1,\ldots ,T_d).$
\end{enumerate}
\end{lem}
 
Given a generalised $d$-hypersurface $\mathbb{S} = (S, \sigma, v)$, we can define a positive (semi-)definite matrix $T_{\mathbb{S}}$, the covariance matrix of $\mathbb{S}$, by
$$ T_\mathbb{S} = \int_S v(x) \otimes v(x) {\rm d} \sigma(x)$$
where, as usual, $\langle (v \otimes v) \xi , \eta \rangle := \langle v, \xi \rangle \langle v , \eta \rangle$ for vectors $\xi, \eta \in \mathbb{R}^d$. As is readily checked, $T_{\mathbb{S}^{d-1}} = \omega_d I_d$. If now $\mathbb{S}_1, \dots , \mathbb{S}_d$ are generalised $d$-hypersurfaces, their covariance matrices satisfy
\begin{equation}\label{eq:discr-1}D(T_{\mathbb{S}_1},\ldots ,T_{\mathbb{S}_d})=\frac{1}{d!}\int_{S_d}\cdots\int_{S_1}|v_1(x_1) \wedge\ldots \wedge v_d(x_d)|^2{\rm d}\sigma_1(x_1)\cdots {\rm d}\sigma_d(x_d).\end{equation}
This is proved in \cite[relation (2.3)]{MR2152912} when $\mathbb{S} = (\mathbb{S}^{d-1}, \mu, n)$ for a nonnegative Borel measure $\mu$  supported on $\mathbb{S}^{d-1}$ and $n$ the unit normal at $x$, but the argument there goes through in our more general setting.

Now we have
$$ Q_j^2(\mathbb{S}_1,\ldots,\mathbb{S}_j){^{2j}} = \int_{S_1}\cdots\int_{S_j}|v_1(x_1) \wedge\ldots \wedge v_j(x_j)|^2 {\rm d}\sigma_1(x_1)\cdots {\rm d}\sigma_j(x_j) $$
$$= c_{d,j} \int_{\mathbb{S}^{d-1}} \dots \int_{\mathbb{S}^{d-1}}\int_{S_j}\cdots\int_{S_1}|v_1(x_1) \wedge\ldots \wedge v_d(x_d)|^2{\rm d}\sigma_1(x_1)\cdots {\rm d}\sigma_j(x_j) {\rm d}\sigma(x_{j+1}) \cdots  {\rm d}\sigma(x_{d})  $$ 
where the constant $c_{d,j}$ is given by
$$ c_{d,j} := \frac{\int_{\mathbb{S}^{d-1}  \times \dots \times \mathbb{S}^{d-1}}|x_1 \wedge\ldots \wedge x_j|^2 {\rm d}\sigma(x_1)\cdots {\rm d}\sigma(x_j) }{\int_{\mathbb{S}^{d-1}  \times \dots \times \mathbb{S}^{d-1}}|x_1 \wedge\ldots \wedge x_d|^2{\rm d}\sigma(x_1)\cdots {\rm d}\sigma(x_{d}) } = \frac{Q_j^2(\mathbb{S}^{d-1})^{2j}}{Q_d^2(\mathbb{S}^{d-1})^{2d}}.$$
 
Therefore:
\begin{prop}\label{prop:mixeddiscr}
With $c_{d,j}$ as above, we have $$Q_j^2(\mathbb{S})^{2j} = c_{d,j} d! D(T_{\mathbb{S}},\ldots ,T_{\mathbb{S}},I_d,\ldots,I_d)$$
where there are $j$ copies of $T_{\mathbb{S}}$ and $(d-j)$ copies of $I_d$.
\end{prop}
Since $c_{d,d}=1$, from this one may read off that
$Q_d^2(\mathbb{S}^{d-1})^{2d} = d! \omega_d^d$.

\subsection{The Aleksandrov--Fenchel and Aleksandrov inequalities}
What drives Theorem~\ref{logcvx} are the Aleksandrov--Fenchel (for the case $p=1$) and Aleksandrov inequalities (for the case $p=2$).
We now proceed to state them.

\begin{thm}[Aleksandrov--Fenchel]\cite[Theorem~7.3.1]{SChneider}\label{thm:AF ineq}
	For non-empty, compact, convex subsets $K_1,\dots,K_d \subseteq \R^d$ the following inequality holds:
	\begin{align*}
		V(K_1,K_2,\dots,K_d)^2 \ge V(K_1,K_1,K_3,\dots,K_d)V(K_2,K_2,K_3,\dots,K_d).
	\end{align*}
\end{thm}

\begin{thm}[Aleksandrov]\cite
[Theorem~5.5.4]{SChneider}
    Let $A_i$, $i=1,\ldots, d$ be  positive definite $d \times d$ matrices. %in $GL_n$. 
    Then, 
    $$D(A_1, A_2, \ldots, A_d)^2\geq D(A_1,A_1,A_3,\ldots, A_d)D(A_2,A_2,A_3, \ldots, A_d).$$
\end{thm}
We are now ready to give the proofs of Theorems~\ref{main_baby} and \ref{logcvx}.

\section{Proof of Theorems~\ref{main_baby} and \ref{logcvx}} \label{proof}
In this short section we complete the proof of Theorem~\ref{logcvx}, condition~\eqref{base} and of Theorem~\ref{main_baby}. Let $p=1$ or $2$.

\medskip
\noindent
Let, for $0 \leq j \leq d$, 
$$a_j(\mathbb{S}) = \left(\frac{Q_j^{p}(\mathbb{S})}{Q_j^{p}(\mathbb{S}^{d-1})}\right)^j.$$ 
In order to establish Theorem~\ref{logcvx} and condition~\eqref{base}, we need to show that 
\begin{equation}\label{general_again}
a_{j-1}(\mathbb{S}) a_{j+1} (\mathbb{S})\leq a_j(\mathbb{S})^2 \mbox{ for } 1 \leq j \leq d-1.
\end{equation}

\medskip
\noindent
We will do this for the case $p=1$. We have, by Theorem~\ref{thm:schneider_digested},
$$ Q_j^1(\mathbb{S})^j = c_{d,j} V(\Pi(\mathbb{S})[j], \mathbb{B}^d[d-j]),$$
where $c_{d,j} =d!/2^j (d-j)! \omega_{d-j}$, so that 
$$ a_j(\mathbb{S}) = \frac{ Q_j^1(\mathbb{S})^j}{ Q_j^1(\mathbb{S}^{d-1})^j} = \frac{V(\Pi(\mathbb{S})[j], \mathbb{B}^d[d-j])}{V(\Pi(\mathbb{S}^{d-1})[j], \mathbb{B}^d[d-j])}
$$
\medskip
\noindent
We have by the Aleksandrov--Fenchel inequality
that the mixed volumes
$$ b_j = b_j(\mathbb{S}):= {V(\Pi(\mathbb{S})[j],\mathbb{B}^d[d-j])}$$
(for $0 \leq j \leq d$) form a log-concave sequence, and moreover, by the multilinear property of mixed volumes, 
$$b_j(\mathbb{S}^{d-1})^2 = b_{j-1}(\mathbb{S}^{d-1})b_{j+1}(\mathbb{S}^{d-1})$$ 
for $1 \leq j \leq d-1$. We deduce that the ratios
$$ a_j(\mathbb{S}) := \frac{b_j(\mathbb{S})}{b_j(\mathbb{S}^{d-1})} = \left(\frac{Q^{1}_j(\mathbb{S})}{Q^{1}_j(\mathbb{S}^{d-1})} \right)^j$$
also form a log-concave sequence, establishing Theorem~\ref{logcvx} and condition~\eqref{base}, and thus %the first assertion of 
Theorem~\ref{main_baby} when $p=1$.

\medskip
\noindent
When $p=2$ we use essentially the same argument, with the Aleksandrov inequalities and Proposition~\ref{prop:mixeddiscr} in place of the Aleksandrov--Fenchel inequalities and Corollary~\ref{cor:schneider_digested}.

\medskip
\noindent
As mentioned in the introduction, the case $p = \infty$ of Theorem~\ref{main_baby} is already covered by Theorem~\ref{thm:general}.

\medskip
\noindent
The proofs of Theorems~\ref{logcvx} and \ref{main_baby} are now complete.

\section{Off-diagonal versions of Propositions~\ref{prop:vislines} and \ref{prop:visplanes}}\label{sec:convex}
Notice that an immediate consequence of Proposition~\ref{prop:vislines} is that
$$ {\rm vis}(\mathbb{S}) \sim_d \inf_{\{u_1, \dots , u_d\} \, {\rm orthonormal}} \left(\sigma(u_1, \mathbb{S}) \dots \sigma(u_d, \mathbb{S})\right)^{1/d},$$
and that the infimum is essentially achieved by the orthonormal basis consisting of the principal directions of the John ellipsoid for $K(\mathbb{S})$.
In this section, we prove the following partial generalisation of Proposition~\ref{prop:vislines}: 
\begin{prop}\label{basic_interp_gen}
For $1 \leq j \leq d$ we have
$$ Q_j^1(\mathbb{S}_1, \dots , \mathbb{S}_j)^j\sim_{d,j}
\inf_{1 \leq m \leq j, \, \{e^m_1, \dots , e^m_d\} \, {\rm orthonormal}}\sum_{\tau \in S_d} {\sigma}(e^1_{\tau(1)},\mathbb{S}_1) \dots {\sigma}(e^j_{\tau(j)},\mathbb{S}_j)
|e^1_{\tau(1)} \wedge \dots \wedge e^d_{\tau(d)}|$$
where for each $m$ such that $j+1 \leq m \leq d$, $\{e^m_1, \dots , e^m_d\}$ is an arbitrary orthonormal basis of $\mathbb{R}^d$. Moreover, essentially extremising orthonormal bases $\{e^m_i\}_{i=1}^d$ for $1 \leq m \leq j$ are given by the principal directions of the John ellipsoid of $\Pi\mathbb{S}_m$. 
\end{prop}

\medskip
\noindent
We will first recall two further properties of the mixed volume that we shall need. These can be found in \cite[Chapter~5]{SChneider}. These are:
 (i) monotonicity: if $K_i \subseteq L_i$ for all $i$, we have $V(K_1, \dots , K_d) \leq 
V(L_1, \dots , L_d)$; and (ii) the formula for mixed volumes of arbitrarily-oriented rectangular boxes: if $1 \leq m \leq d$, and $\mathcal{R}_m$ is a rectangular box with edges of lengths 
$a^m_i$ in orthonormal directions $\{f^m_1, \dots , f^m_d\}$ for $1 \leq m \leq d$, then 
\begin{equation}\label{eq:sophisticatedboxes}
V(\mathcal{R}_1, \dots , \mathcal{R}_d) = \frac{1}{d!} \sum_{\tau \in S_d} a^{1}_{\tau(1)} \dots a^{d}_{\tau(d)} |f^1_{\tau(1)} \wedge \dots \wedge f^d_{\tau(d)}|.
\end{equation}
For a derivation and discussion of this formula see \cite[Theorem~5.3.2]{SChneider}, (the top of p.~304). 

\medskip
\noindent
We begin the argument for the proposition. Let $1 \leq m \leq j$. Suppose that the John ellipsoid $\mathcal{E}_m$ of $\Pi\mathbb{S}_m$ has principal directions $\{e^m_1, \dots , e^m_d\}$ and semiaxes of lengths $\sim_d \sigma(e^m_i,\mathbb{S}_m)$. Using Theorem~\ref{thm:schneider_digested} and monotonicity of mixed volume  we can deduce that 
$$ Q_j^1(\mathbb{S}_1, \dots , \mathbb{S}_j)^j \sim_{d,j} V( \Pi \mathbb{S}_1, \dots , \Pi \mathbb{S}_j, \mathbb{B}^d [d-j]) \sim_{d,j} V(\mathcal{E}_1, \dots, \mathcal{E}_j,  \mathbb{B}^d [d-j]).$$
Now for any choices of orthonormal bases $\{f^m_1, \dots , f^m_d\}$, ($1 \leq m \leq j$) for $\mathbb{R}^d$ we have, using the polarity relation Lemma~\ref{lem:polarity}, $\sigma(f^m_i, \mathbb{S}_m)^{-1}f^m_i \in K(\mathbb{S}_m)$ for $1 \leq i \leq d$,
and thus, for a suitable constant $C_{d,j}$, we have that $C_{d,j} \Pi(\mathbb{S}_m)$ is contained in the ellipsoid $\mathcal{F}_m$ with principal directions 
$\{f^m_1, \dots , f^m_d\}$ and semiaxes of lengths $\sim \sigma(f^m_1,\mathbb{S}_m), \dots , \sigma(f^m_d,\mathbb{S}_m)$ respectively.
We therefore deduce (using monotonicity once again) that 
$$ Q_j^1(\mathbb{S}_1, \dots , \mathbb{S}_j)^j \lesssim_{d,j} V(\mathcal{F}_1, \dots, \mathcal{F}_j,  \mathbb{B}^d [d-j])$$
and moreover that we have essential equality here if $\{f^m_1, \dots , f^m_d\}$ are orthonormal bases for the John ellipsoid for $\Pi\mathbb{S}_m$ for $1 \leq m \leq j$. So we need to calculate 
$$ V(\mathcal{F}_1, \dots, \mathcal{F}_j,  \mathbb{B}^d [d-j])$$
or, what is essentially the same thing,
$$ V(\mathcal{R}_1, \dots, \mathcal{R}_j,  \mathbb{B}^d [d-j])$$
where, for $1 \leq m \leq j$, $\mathcal{R}_m$ is a rectangular box with edges of lengths 
$\sim \sigma(f^m_j, \mathbb{S}_m)$ in directions $\{f^m_1, \dots , f^m_d\}$. Indeed, by applying \eqref{eq:sophisticatedboxes} we obtain
\begin{equation*}\label{boxes_again}
V(\mathcal{R}_1, \dots, \mathcal{R}_j,  \mathbb{B}^d [d-j])
\sim_{d,j} \sum_{\tau \in S_d}{\sigma}(f^1_{\tau(1)},\mathbb{S}_1) \dots {\sigma}(f^j_{\tau(j)},\mathbb{S}_j)|f^1_{\tau(1)} \wedge \dots \wedge f^d_{\tau(d)}|
\end{equation*}
where, for each $m$ such that $j+1 \leq m \leq d$, $\{f^m_1, \dots , f^m_d\}$ is an arbitrary orthonormal basis of $\mathbb{R}^d$. 

\medskip
\noindent
We therefore deduce that 
$$  Q_j^1(\mathbb{S}_1, \dots , \mathbb{S}_j)^j \lesssim_{d,j}  \sum_{\tau \in S_d}{\sigma}(f^1_{\tau(1)},\mathbb{S}_1) \dots {\sigma}(f^j_{\tau(j)},\mathbb{S}_j)|f^1_{\tau(1)} \wedge \dots \wedge f^d_{\tau(d)}|$$
where, for each $m$ such that $1 \leq m \leq d$, $\{f^m_1, \dots , f^m_d\}$ is an arbitrary orthonormal basis of $\mathbb{R}^d$. Moreover there is essential equality here if for $1 \leq m \leq j$, 
$\{f^m_1, \dots , f^m_d\}$ is an orthonormal basis for $\mathcal{E}_m$.

\medskip
\noindent
These considerations conclude the argument.

\medskip
\noindent
Similarly, there is a corresponding generalisation of the ``orthogonal" case of Proposition~\ref{prop:visplanes}, which we invite the reader to formulate and prove.

\section{Further contextual remarks}\label{sec:five}
\subsection{Harmonic Analysis} Considerable interest has recently arisen in the harmonic analysis literature in the quantities $Q_j^p(\mathbb{S})$ which we have considered. The reader is referred to \cite{Guth}, \cite{Carbery--Valdimarsson}, \cite{Zhang} and \cite{Z-K} in connection with how these quantities arise in multilinear Kakeya problems, in which case the value $p=1$ is relevant, and to \cite{Gressman1}, \cite{Gressman2}, \cite{Gressman3} in the context of ``non-concentration" inequalities which arise in the construction of affine-invariant measures on surfaces in Euclidean spaces, where the interaction between the values $p=1$ and $p = \infty$ is also significant (see \cite{Gressman2}), and where lower bounds on $Q_j^p(\mathbb{S})$ are of interest.

\medskip
\noindent
In multilinear Kakeya setting, Guth \cite{Guth} introduced the quantities (termed ``directional volumes")
$$ \int_S |n(x) \cdot e| {\rm d}\sigma(x)$$ 
where $S$ is a hypersurface in $\mathbb{R}^d$ and $e \in \mathbb{R}^d$ is a unit vector, and established a critical role for them in the argument for the endpoint multilinear Kakeya theorem; with the benefit of hindsight coming from convex geometry one can see that the directional volumes, and more generally the approach using visibility, do indeed arise naturally in connection with the natural convex geometric background outlined in  Section~\ref{sec:refined} above. For the applications to multilinear Kakeya, (cf. \cite{Zhang} and \cite{Gressman4}), both upper and lower bounds on $Q_d^1(\mathbb{S}) \sim {\rm vis}(\mathbb{S})$ are needed. The lower bounds are those given by Propositions~\ref{prop:vislines} and \ref{prop:visplanes}. The upper bounds are obtained by applying B\'ezout's theorem to certain Crofton-like formulae for $Q_d^1(\mathbb{S}_1,  \dots ,\mathbb{S}_d)$ when $\mathbb{S}_j$ are classical smooth hypersurfaces, and which express them in terms of counts of incidences of translates of hypersurfaces, see Section~\ref{sec:crofton} below.  

\medskip
\noindent
The role played by convex geometry has hitherto been somewhat buried in some of the multilinear Kakeya literature; we hope that the exhumation carried out in this paper will help clarify how convex geometry crucially underpins some of the considerations in the development of that theory. 

\subsection{Crofton-like formulae and B\'ezout's theorem}\label{sec:crofton} 
Crofton-like formulae are well-known within the integral geometry and convex geometry communities -- ``mixed volumes count intersections". See for example \cite[Section 6.4]{SchW}. In our context, the following is relevant. Suppose that $p=1$ and that each $v_i(x_i)$ is the unit normal $n(x_i)$ to a classical smooth hypersurface $S_i$ at $x_i$. Then we have the formula 
\begin{equation}\label{twentyfive}
\begin{aligned}
Q_d^1(S_1, \dots, S_d)^d
&= \int_{S_d} \dots  \int_{S_1} |{n}(x_1) \wedge \dots 
\wedge {n}(x_{d})| {\rm d}\sigma (x_1) \dots {\rm d}\sigma(x_{d}) \\
&= \int_{\mathbb{R}^d} \cdots \int_{\mathbb{R}^d} \#\left[B \cap (S_1 - \xi_1) 
\cap \dots \cap  (S_d - \xi_d)\right] {\rm d}\xi_1 \dots {\rm d}\xi_d\\
&= \int_{\mathbb{R}^d} \cdots \int_{\mathbb{R}^d} \#\left[S_1 \cap (S_2 - \xi_2) 
\cap \dots \cap  (S_d - \xi_d) \right] {\rm d}\xi_2 \dots {\rm d}\xi_d
\end{aligned}
\end{equation}
where $B$ is some ball in $\mathbb{R}^d$ of unit volume. 

\medskip
\noindent
Such Crofton-like formulae were re-derived in an {\em ad hoc} way by Zhang \cite{Zhang} for his treatment of the $k_j$-plane case of multilinear Kakeya. We have been unable to find a simple self-contained argument for this formula in the literature, so for the convenience of the reader, we offer an elementary argument for it in the appendix. 

\medskip
\noindent
The identity \eqref{twentyfive} above
leads to the upper bound
\begin{equation}\label{eq:rkx}
\begin{aligned}
Q_d^1(S_1, \dots , S_d)^d 
\leq {\rm ess} \max &\#\left[S_1 \cap (S_2 - \xi_2) 
\cap \dots \cap  (S_d - \xi_d) \right] \times \\
&\times |\{(\xi_2, \dots , \xi_d) \, : \, S_1 \cap (S_2 - \xi_2) 
\cap \dots \cap  (S_d - \xi_d) \neq \emptyset\}|.
\end{aligned}
\end{equation}
If on a positive proportion of $\{(\xi_2, \dots , \xi_d) \, : \, S_1 \cap (S_2 - \xi_2) 
\cap \dots \cap  (S_d - \xi_d) \neq \emptyset\}$ the function $\#\left[S_1 \cap (S_2 - \xi_2) \cap \dots \cap  (S_d - \xi_d) \right]$ is comparable to its essential maximum, this estimate is sharp, and hence the right-hand side of \eqref{eq:rkx} is also (up to constants) a lower bound in such a situation.

\medskip
\noindent
This can be used in conjunction with B\'ezout's theorem. Indeed, if $S_i$ is a subset of an algebraic hypersurface of degree $m_i$ and $S_i$ has finite surface area, finiteness of the first integral in \eqref{twentyfive} shows that for almost every $(\xi_2, \dots ,\xi_d)$ the quantity 
$\#\left[S_1 \cap (S_2 - \xi_2) 
\cap \dots \cap  (S_d - \xi_d) \right]$ is finite, so by B\'ezout's theorem we may conclude that it is almost-everywhere at most $m_1 \dots m_d$.

\medskip
\noindent
In the convex geometry literature, results with the same flavour as Theorem~\ref{thm:general}  are often described as B\'ezout theorems, cf. \cite{MR3632081}. 

\subsection{An isoperimetric-type inequality}
An immediate consequence of Theorem~\ref{schneider_digested_baby}, Theorem~\ref{main} and the observation that $Q_1^1(\mathbb{S}) = \sigma(\mathbb{S})$ is the fact that the quantity 
$$ \frac{{\rm vol}( \Pi \mathbb{S})^{1/d}}{\sigma(\mathbb{S})}$$ 
is maximised over all generalised $d$-hypersurfaces by the unit sphere with its surface measure and the Gauss map; in other words,
\begin{equation}\label{eq:isoper}\frac{{\rm vol}( \Pi \mathbb{S})^{1/d}}{\sigma(\mathbb{S})}\leq  \frac{{\rm vol}( \Pi \mathbb{S}^{d-1})^{1/d}}{\sigma(\mathbb{S}^{d-1})}.
\end{equation}
This is well-known when $\mathbb{S}$ is the surface of a convex body. In this case, it is a simple consequence of Urysohn's inequality, see for example \cite[Theorem 1.5.11]{GianMilArt}. Indeed, note that Urysohn's inequality says that for a convex body $K\subset\mathbb{R}^d$, the ratio $\frac{({\rm vol}K)^{1/d}}{w(K)}$, where $w(K)$ is the mean-width of $K$, is maximised when $K$ is a ball. Applying this to the projection body and taking into account that $w(\tilde\Pi K)$ is a multiple of the surface area of $K$, we obtain \eqref{eq:isoper}.
Moreover, Theorem~\ref{main} provides a chain of inequalities relating a family of quantities whose extreme members are 
$${\rm vol}( \Pi \mathbb{S})^{1/d} \mbox{  and } \frac{{\rm vol}( \Pi \mathbb{S}^{d-1})^{1/d}}{\sigma(\mathbb{S}^{d-1})}\sigma(\mathbb{S}).$$ 
Inequality \eqref{eq:isoper} can be thought of as a kind of isoperimetric inequality, especially in view of the fact that for a convex body $L$, ${\rm vol}( \Pi L) \sim_d ({\rm vol} \,L)^{d-1}$. More precisely, from the works of Petty \cite{Petty} (for a proof see also in \cite[Theorem 9.2.9]{Gardner}) and Zhang \cite[Theorem 2]{ZhangProj} we know that $$\frac{1}{d^d}\binom{2d}{d}\leq{\rm vol}( (\Pi L)^{*})({\rm vol} \,L)^{d-1}\leq\left(\frac{\omega_d}{\omega_{d-1}}\right)^d.$$ 
Using \eqref{Santalo} we get upper and lower bounds for the ratio
$$\frac{({\rm vol} \,L)^{d-1}}{{\rm vol}( \Pi L)}.$$
Note that for this ratio it is an open problem to find the sharp upper and lower bounds, see \cite[Problem 9.1]{Gardner}.

\subsection{Geometric probability} One may also naturally interpret Theorem~\ref{main} as an extremal result for the expected $j$-volume of a random rhomboid/parallelotope or simplex in $\mathbb{R}^d$ with $j$ edges chosen according to some probability distribution.

\medskip
\noindent
Indeed, we take $\mathbb{S}$ to be $\mathbb{R}^{d}$ together with some probability measure $\mu$, and $v:  \mathbb{R}^{d} \to \mathbb{R}^{d}$ to be the identity map $\iota$. Suppose we choose vectors $\xi_1, \dots , \xi_j$ according to the probability distribution $\mu$, and then consider the $j$-dimensional volume $V_j(\xi_1, \dots , \xi_j)$ of the polytope with these edges. 
Then we have
$$  \mathbb{E}^\mu (V_j^p) = \int_{\mathbb{R}^{d}} \cdots \int_{\mathbb{R}^{d}} |\xi_1 \wedge \dots \wedge \xi_j|^p {\rm d} \mu(\xi_1) \dots {\rm d} \mu(\xi_j) = 
%Q_j^1(\mathbb{S}_1, \dots , \mathbb{S}_j)^j$$
Q_{j}^{p}(\mathbb{S})^{jp}$$
with $\mathbb{S} = (\mathbb{R}^{d}, \mu, \iota)$.  To ensure that these quantities are finite, we shall assume that $\mathbb{E}^\mu(V^p_1) = \int_{\mathbb{R}^d} |\xi|^p {\rm d} \mu (\xi)< \infty$.
 Let $\sigma$ be the uniform probability distribution on $\mathbb{S}^{d-1}$ and let $\eta_i \in \mathbb{S}^{d-1}$ be chosen according to $\sigma$, with corresponding $j$-volume $W_j(\eta_1, \dots , \eta_j)$.  We wish to compare the quantities $\mathbb{E}^\mu(V^p_j)$ and $\mathbb{E}^\mu(W^p_j)$. In order for this comparison to be meaningful even for $j=1$, it is natural to further assume that $\mathbb{E}^\mu(V^p_1) = \int_{\mathbb{R}^d} |\xi|^p {\rm d} \mu (\xi)=1$, in which case we have $\mathbb{E}^\mu(V^p_1)= \mathbb{E}^\mu(W^p_1)$. Theorem~\ref{main} then immediately gives, for $1 \leq j \leq d-1$, and $p=1,2$,
$$ \left(\frac{\mathbb{E}^\mu(V_{j+1}^p)}{\mathbb{E}^\sigma( W_{j+1}^p)}\right)^{1/(j+1)} 
\leq \left(\frac{\mathbb{E}^\mu(V_{j}^p)}{\mathbb{E}^\sigma( W_{j}^p)}\right)^{1/j} 
\leq \dots 
\leq \frac{\mathbb{E}^\mu(V^p_1)}{\mathbb{E}^\sigma(W^p_1)} =\frac{\int_{\mathbb{R}^d}|\xi|^p {\rm d} \mu(\xi)}{\int_{\mathbb{R}^d}|\eta|^p {\rm d} \sigma(\eta)}.$$ 
In particular, we then have:

\begin{cor}
Let $\mu$ be a probability distribution on $\mathbb{R}^d$ and $\sigma$ be the uniform probability distribution on $\mathbb{S}^{d-1}$. Let $p =1$ or $2$ and suppose that $\int_{\mathbb{R}^d} |\xi|^p {\rm d} \mu =1$. Then, for $1 \leq j \leq d$, 
$$ \mathbb{E}^\mu(V^p_j)
\leq \mathbb{E}^\sigma(W_j^p).$$
\end{cor}
If we combine this with the explicit calculations of $Q_j^1(\mathbb{S}^{d-1})$ mentioned earlier, we see that 
$$  \mathbb{E}^\mu(V_j^1) \leq \mathbb{E}^\sigma(W_j^1)=\frac{\omega_{d-1}^j d!}{\omega_d^{j-1} \omega_{d-j} d^j (d-j)!}.$$

The (most interesting) case $j=d$ of this has been noted previously by R.~Vitale, \cite[p.~203]{microscopy}, who also observed that 
$$ \mathbb{E}^\sigma(W_d^1)^{1/d} = \left(\frac{\omega_{d-1}^d d!}{\omega_d^{d-1}  d^d}\right)^{1/d} \leq \left(\frac{d!}{d^d}\right)^{1/2d},$$
and that both sides tend to $e^{-1/2}$ as $ d \to \infty$. Similarly, in the case $p=2$ we obtain
$$ \mathbb{E}^\mu(V_d^2) \leq \mathbb{E}^\sigma(W_d^2) = \frac{Q^2_d(\mathbb{S}^{d-1})^{2d}}{\sigma_{d-1}^{d}} = \frac{d! \omega_d^d}{(d \omega_d)^{d}} = \frac{d!}{d^d},$$
whose $d$'th root tends to $e^{-1}$ as $d \to \infty$.
So the ratio $\mathbb{E}^\sigma(W_d^1)^{1/d}/
\left(\mathbb{E}^\sigma(W_d^2)\right)^{1/2d}$ (which is at most $1$ by the Cauchy--Schwarz inequality) in fact approaches $1$ as $d \to \infty$.

\medskip
\noindent
See also Miles \cite{miles_1971} for further connections with random volumes.

\subsubsection{A weak form of Theorem~\ref{main} for $p \neq 1,2$}\label{sec:genp}
When the underlying measure is log-concave, we can deduce a certain weak analogue for $p \neq 1,2$ of Theorem~\ref{main} in this setting from  Theorem~\ref{main} itself. Theorem~\ref{main} tells us that for $p=1,2$ we have
$$ \frac{\mathbb{E}^\mu(V^p_{j+1})^{1/p(j+1)}}{\mathbb{E}^\sigma(W^p_{j+1})^{1/p(j+1)}} 
\leq \frac{\mathbb{E}^\mu(V^p_{j})^{1/pj}}{\mathbb{E}^\sigma(W^p_j)^{1/pj}}.$$ 
If $\mu$ is a log-concave probability measure and $p \neq 1,2$ we can show:
\begin{prop}\label{prop:weak}
Let $\mu$ be a log-concave probability measure on $\mathbb{R}^d$. Then, for $p \geq 1$ we have
    $$ \frac{\mathbb{E}^\mu(V^p_{j+1})^{1/p(j+1)}}{\mathbb{E}^\sigma(W^p_{j+1})^{1/p(j+1)}} 
\leq C p^2 \sqrt{\frac{d+1}{d+p}} \frac{\mathbb{E}^\mu(V^p_{j})^{1/pj}}{\mathbb{E}^\sigma(W^p_j)^{1/pj}}
\leq C_d p^{3/2} \frac{\mathbb{E}^\mu(V^p_{j})^{1/pj}}{\mathbb{E}^\sigma(W^p_j)^{1/pj}}$$ 
where $C$ is an absolute constant. 
\end{prop}
Note that we do not need the condition $\int_{\mathbb{R}^d} |\xi|^p {\rm d}\mu(\xi)=1$ here. This result will require some reverse H\"older inequalities for which we now set the scene.

\medskip
\noindent
Let $f$ be a seminorm defined on $\mathbb{R}^d$. For any $1 \leq q <p <\infty$, we have the following reverse H\"{o}lder inequalities:

\begin{itemize}
\item Let $\mu$ be an arbitrary nondegenerate log-concave probability measure $\mu$ on $\mathbb{R}^d$. Then there is some absolute constant $c>1$ such that for $1 \leq q < p <\infty$, 
$$ \|f\|_{L^q({\rm d} \mu)} \leq \|f\|_{L^p({\rm d} \mu)} \leq \frac{cp}{q} \|f\|_{L^q({\rm d} \mu)},$$
see \cite[Theorem 3.5.11, p. 115]{GianMilArt}

\item Let $\sigma$ be the uniform {\em probability} measure on the sphere. Then for some absolute constant $c > 1$,
$$\left(\int_{\mathbb{S}^{d-1}}|f(x)|^p\, {\rm d}\sigma(x)\right)^{1/p}\leq \frac{cp}{q}\sqrt{\frac{d+q}{d+p}}\left(\int_{\mathbb{S}^{d-1}}|f(x)|^q\, {\rm d}\sigma(x)\right)^{1/q},$$
see \cite[p. 116]{GianMilArt}.
\end{itemize}

From these we derive some related reverse H\"older inequalities which will be useful for us. Let $\mu$ be any probability measure on $\mathbb{R}^d$. Fix $1 \leq j \leq d$. It is easy to check that for each $k \leq j$, (for fixed $\xi_{k+2}, \dots , \xi_j$), 
$$\xi_{k+1}\mapsto \left(\int_{\mathbb{R}^{d}}\cdots\int_{\mathbb{R}^{d}}\left|\xi_1\wedge \xi_2\wedge\ldots\wedge \xi_j\right|^q{\rm d}\mu(\xi_{1})\ldots {\rm d}\mu(\xi_k)\right)^{1/q}$$ 
is a seminorm on $\mathbb{R}^d$. 

\medskip
\noindent
Thus, if in addition $\mu$ is log-concave, the first reverse H\"older inequality above gives, for each $k$,
$$ \left(\int_{\mathbb{R}^{d}}\left(\int_{\mathbb{R}^{d}}\cdots\int_{\mathbb{R}^{d}}\left|\xi_1\wedge \xi_2\wedge\ldots\wedge \xi_j\right|^q{\rm d}\mu(\xi_{1})\ldots {\rm d}\mu(\xi_k)\right)^{p/q}\, {\rm d}\mu(\xi_{k+1})\right)^{1/p}$$
$$\leq \frac{cp}{q} \left(\int_{\mathbb{R}^{d}}\int_{\mathbb{R}^{d}}\cdots\int_{\mathbb{R}^{d}}\left|\xi_1\wedge \xi_2\wedge\ldots\wedge \xi_d\right|^q{\rm d}\mu(\xi_{1})\ldots {\rm d}\mu(\xi_{k+1})\right)^{1/q}.$$
Beginning with the case $k=0$ of this inequality, an inductive argument gives
$$ \left(\int_{\mathbb{R}^{d}}\cdots \int_{\mathbb{R}^{d}}\left|\xi_1\wedge \xi_2\wedge\ldots\wedge \xi_j\right|^p{\rm d}\mu(\xi_1)\ldots {\rm d}\mu(\xi_{k})\right)^{1/p}$$
$$\leq \left(\frac{cp}{q}\right)^k \left(\int_{\mathbb{R}^{d}}\cdots \int_{\mathbb{R}^{d}}\left|\xi_1\wedge \xi_2\wedge\ldots\wedge \xi_j\right|^q{\rm d}\mu(\xi_1)\ldots {\rm d}\mu(\xi_{k})\right)^{1/q}$$
successively for $k=1, \dots, j$.
The final step $k=j$ is
\begin{equation}\label{eq:cvx1}
\mathbb{E}^\mu(V^p_j)^{1/p} \leq \left(\frac{cp}{q}\right)^j 
\mathbb{E}^\mu(V^q_j)^{1/q}.
\end{equation}
Similarly, by using the second reverse H\"older inequality above, we obtain
\begin{equation}\label{eq:cvx2}
\mathbb{E}^\sigma(W^p_j)^{1/p} \leq \left(\frac{cp}{q}\sqrt{\frac{d+q}{d+p}}\right)^j 
\mathbb{E}^\sigma(W^q_j)^{1/q},
\end{equation}
(which can also be read off from the explicit formulae for $Q^p_j(\mathbb{S}^{d-1})$ to which we referred above).

\medskip
\noindent
We use these inequalities (in fact only the case $q=1$) to derive Proposition~\ref{prop:weak}. Indeed, using \eqref{eq:cvx1} we have
$$ \mathbb{E}^\mu(V^p_{j+1})^{1/p(j+1)}   
\leq  cp \, \mathbb{E}^\mu(V^1_{j+1})^{1/(j+1)} \leq 
cp \, \mathbb{E}^\sigma(W^1_{j+1})^{1/(j+1)} \frac{\mathbb{E}^\mu(V^1_{j})^{1/j}}{\mathbb{E}^\sigma(W^1_j)^{1/j}},$$
by the case $p=1$ of Theorem~\ref{main}. 
Therefore,
\begin{align*}
\frac{\mathbb{E}^\mu(V^p_{j+1})^{1/p(j+1)}}{\mathbb{E}^\sigma(W^p_{j+1})^{1/p(j+1)}} 
\leq & cp \, \frac{\mathbb{E}^\sigma(W^1_{j+1})^{1/(j+1)}}{\mathbb{E}^\sigma(W^p_{j+1})^{1/p(j+1)}} \frac{\mathbb{E}^\mu(V^1_j)^{1/j}}{\mathbb{E}^\sigma(W^1_j)^{/j}} \\
\leq  & cp \, \frac{\mathbb{E}^\sigma(W^1_{j+1})^{1/(j+1)}}{\mathbb{E}^\sigma(W^p_{j+1})^{1/p(j+1)}} \frac{\mathbb{E}^\mu(V^p_j)^{1/pj}}{\mathbb{E}^\sigma(W^p_j)^{1/pj}} \frac{\mathbb{E}^\sigma(W^p_j)^{1/pj}}{\mathbb{E}^\sigma(W^1_j)^{1/j}}
\end{align*}
by H\"older's inequality. This in turn equals
$$ cp \, \frac{B(j+1, d,p)}{B(j, d,p)}
\frac{\mathbb{E}^\mu(V^p_{j})^{1/pj}}{\mathbb{E}^\sigma(W^p_j)^{1/pj}} $$ 
where
$$ B(j,d,p) = \frac{\mathbb{E}^\sigma(W^1_{j})^{1/j}}{\mathbb{E}^\sigma(W^p_{j})^{1/pj}}.$$
Clearly $B(j,d,p) \leq 1$, and by \eqref{eq:cvx2} it also satisfies 
$$ B(j,d,p)^{-1} \leq c p\sqrt{\frac{d+1}{d+p}}.$$
Consequently we have
$$ \frac{\mathbb{E}^\mu(V^p_{j+1})^{1/p(j+1)}}{\mathbb{E}^\sigma(W^p_{j+1})^{1/p(j+1)}} 
\leq c^2 p^2 \sqrt{\frac{d+1}{d+p}} \frac{\mathbb{E}^\mu(V^p_{j})^{1/pj}}{\mathbb{E}^\sigma(W^p_j)^{1/pj}}.$$

\subsection{Final Remarks}
More recently, the quantities $Q^p_d(\mathbb{S}^{d-1})$ for general $p$ have also arisen in the work of Bennett and Tao \cite{bennett-tao} on adjoint Brascamp--Lieb inequalities.

\medskip
\noindent
Some results related to ours, but for the case $p<0$, have recently been considered by Milman and Yehudayoff \cite{milman2022sharp}, see especially Theorems 7.2 and 1.7. 

\section{Appendix. The Crofton formula: $Q_d^1$ and counting intersections}\label{appendix:intersections}
Here we offer an elementary, calculus-based approach to the Crofton-like formula \eqref{twentyfive}. The argument is presented with simplicity in mind, not maximum generality. It is essentially a simplification of a special case of Zhang's argument from \cite{Zhang}. See also \cite[Lemma~3.2]{Z-K} and \cite{Gressman4}.
 
 \begin{prop}\label{int_geom}
Suppose that $S_1, S_2, \dots , S_d$ are hypersurfaces in $\mathbb{R}^d$. Let $B$ 
denote a ball of unit volume in $\mathbb{R}^d$. Then
\begin{equation}\label{twentysix}
\begin{aligned}
&\int_{S_d} \dots  \int_{S_1} |{n}(x_1) \wedge \dots 
\wedge {n}(x_{d})| {\rm d}\sigma(x_1) \dots {\rm d}\sigma(x_{d}) \\
&= \int_{\mathbb{R}^d} \cdots \int_{\mathbb{R}^d} \#\left[B \cap (S_1 - v_1) 
\cap \dots \cap  (S_d - v_d)\right] {\rm d}v_1 \dots {\rm d}v_d\\
&= \int_{\mathbb{R}^d} \cdots \int_{\mathbb{R}^d} \#\left[S_1 \cap (S_2 - v_2) 
\cap \dots \cap  (S_d - v_d) \right] {\rm d}v_2 \dots {\rm d}v_d.
\end{aligned}
\end{equation}

\end{prop}
\begin{proof}
The second and third expressions are easily seen to be equal using only the fact that $B$ has unit volume.\footnote{Indeed, if $S \subseteq \mathbb{R}^d$ is a finite set and $B \subseteq \mathbb{R}^d$ satisfies $|B| =1$, then $\int \chi_{B-v}(x) {\rm d}v = \int \chi_{B-x}(v) {\rm d}v = |B-x| = |B| =1$,
so that $\# S = \sum_x \chi_S(x) = \sum_x \chi_S(x) \int \chi_{B-v}(x) {\rm d}v = \int \sum_x \chi_{(B-v) \cap S}(x){\rm d}v  = \int \#[(B-v) \cap S] {\rm d}v = \int \#[B\cap (S-v)] {\rm d}v.$} Therefore we concentrate on equality of the first and second.

\medskip
\noindent
We note that both expressions are (countably) multi-additive functionals of $S_1, \dots, S_d$ so it suffices to establish the identity for ``small" hypersurfaces $S_i$, each of which we may assume has a $C^1$ injective parametrisation 
$$  x_i : U_i \to \mathbb{R}^d$$
where $U_i \subseteq \mathbb{R}^{d-1}$ is an open set. 

\medskip
\noindent
Let 
$$ \Phi: B \times U_1 \times \dots \times U_d \to \mathbb{R}^d \times \dots \times \mathbb{R}^d$$ 
be given by
$$ \Phi(u_0, u_1, \dots , u_d) = \Bigl(x_1(u_1) - u_0, \dots, x_d(u_d) - u_0\Bigr).$$

\medskip
\noindent
{\bf Claim.} For each $v \in \mathbb{R}^d \times \dots \times \mathbb{R}^d$, the mapping
$$ u = (u_0, \dots , u_d) \mapsto u_0$$
is a bijection between the sets $$\{ u \in B \times U_1 \times \dots \times U_d \, : \, \Phi (u) = v\} \mbox{ and }B \cap (S_1 - v_1) 
\cap \dots \cap  (S_d - v_d).$$

\medskip
\noindent
{\em Proof of Claim.} 
\medskip
\noindent
The injectivity follows directly from that of each $x_i$. For each $u = (u_0, u_1, \dots , u_d) \in B \times U_1 \times \dots \times U_d$ such that $\Phi (u) = v$, consider its first component $u_0 \in B$: since $\Phi(u) = v$ we have that $u_0 = x_i(u_i)-v_i$ for each $1 \leq i \leq d$, so that $u_0 \in  B \cap (S_1 - v_1) 
\cap \dots \cap  (S_d - v_d)$. On the other hand, if $w_0 \in B \cap (S_1 - v_1) 
\cap \dots \cap  (S_d - v_d)$, then there exist $u_i \in U_i$ for $1 \leq i \leq d$ such that $x_i(u_i) - v_i = w_0$, meaning that $\Phi(w_0, u_1, \dots u_n) = v$. This clearly establishes the desired bijection. \hfill \qed

\medskip
\noindent
Therefore, the middle expression in \eqref{twentysix} equals
$$  \int_{\mathbb{R}^d} \cdots \int_{\mathbb{R}^d} \#\{ u \in B \times U_1 \times \dots \times U_d \, : \, \Phi(u) = v\} {\rm d}v_1 \dots {\rm d}v_d
$$
$$ = \int_B \int_{U_d} \dots \int_{U_1} | \det J_\Phi(u_0, u_1, \dots , u_d)| {\rm d} u_1 \dots {\rm d} u_d {\rm d} u_0$$
by the change of variables formula for multiple integrals.

\medskip
\noindent
We next want to look at the Jacobian matrix of $\Phi$, which has size $d^2 \times d^2$. Arrange the $d$ components $\Phi_j$ of $\Phi$ vertically in a $d^2 \times 1$ column. Focus on the $d$ entries coming from $\Phi_j$. Differentiating $\Phi_j(u) = x_j(u_j) - u_0$ with respect to $u_0$ gives a $d \times d$ block $-I$,
differentiating $\Phi_j$ with respect to $u_k$ for $k \neq j$ gives a $d \times (d-1)$ zero block, and differentiating $\Phi_j$ with respect to $u_j$ gives a $d \times (d-1)$ block $(\nabla x_j)(u_j)$.
So the $j$'th row block of $J_\Phi$ (corresponding to the contributions from $\Phi_j$) has the horizontal block form $( -I \, 0 \, \dots 0  \, (\nabla x_j)(u_j) \, 0 \, \dots \, 0)$. 

\medskip
\noindent
Let $y_j(u_j)$ be the wedge product of the $(d-1)$ columns of $(\nabla x_j)(u_j)$, so that we have $ y_j(u_j) = \pm n(x_j(u_j)) |y_j(u_j)|$, and the element of surface area ${\rm d} \sigma(x_j)$ on $S_j$ is given by $|y_j(u_j)| {\rm d} u_j$. 

\medskip
\noindent
{\bf Claim.} We have
$$|\det J_\Phi(u)|= |y_1(u_1) \wedge \dots \wedge y_d(u_d)|.$$
\medskip
See below for the proof. Hence,
$$|\det J_\Phi(u)|= |y_1(u_1) \wedge \dots \wedge y_d(u_d)|$$
$$=|n(x_1(u_1)) |y_1(u_1)| \wedge \dots \wedge  n(x_d(u_d)) |y_d(u_d)|| $$
$$ = |n(x_1(u_1)) \wedge \dots \wedge  n(x_d(u_d))|\times |y_1(u_1)| \dots  |y_d(u_d)|.$$
Therefore, the middle expression in \eqref{twentysix} equals 
$$ \int_B \int_{U_d} \dots \int_{U_1} | n(x_1(u_1)) \wedge \dots \wedge  n(x_d(u_d))| |y_1(u_1)| \dots |y_d(u_d)| {\rm d} u_1 \dots {\rm d} u_d {\rm d} u_0$$
$$= \int_B \left( \int_{S_d} \dots  \int_{S_1} |{n}(x_1) \wedge \dots 
\wedge {n}(x_{d})| {\rm d}\sigma(x_1) \dots {\rm d}\sigma(x_{d})\right) {\rm d} u_0$$
$$ = \int_{S_d} \dots  \int_{S_1} |{n}(x_1) \wedge \dots 
\wedge {n}(x_{d})| {\rm d}\sigma(x_1) \dots {\rm d}\sigma(x_{d}),$$
since $B$ has unit volume. This completes the proof of Proposition~\ref{int_geom}.
\end{proof}

Before we establish the claim, let us first do some relabelling in order to emphasise its essentially linear-algebraic nature. For each $j$, we relabel the columns of $\nabla x_j(u_j)$ as $a_{jk}$ for $1 \leq k \leq d-1$, and write $A_j$ for the $d \times (d-1)$ matrix whose columns are $a_{j1}, \dots, a_{j (d-1)}$. Call $J_\Phi(u)$ to be $\mathcal{A}$. Then the $d^2 \times d^2$ matrix $\mathcal{A}$ has the block form 

\medskip
$$\mathcal{A} = \begin{pmatrix}
           -I &  A_1   &   0    & \cdots & 0  & 0 & \cdots  & 0  &  0 \\
           -I  & {0}   &   A_2    & \cdots & 0 &    0  & \cdots  &    0    & 0 \\
           & &  & \ddots & & & \\ 
           -I &  0   &  0 & \cdots  & A_{j-1} & 0 & \cdots & 0  & 0 \\
           -I & 0 & 0 & \cdots & 0 & A_j & \cdots & 0 & 0 \\
           & & &  & &  & \ddots \\
           -I & 0 & 0 & \cdots & 0 & 0 & \cdots & 0 & A_d\\
   \end{pmatrix}.$$
Let $y_j = a_{j1} \wedge \dots \wedge a_{j(d-1)}$ be the exterior product\footnote{The definition is recalled below.} of the columns of $A_j$, and let $Y$ be the $d \times d$ matrix whose columns are $y_1, \dots , y_d$. With this set-up, the claim follows immediately from the following multilinear-algebraic lemma:

\begin{lem}\label{lem:determinant}
We have
\begin{equation}\label{eq:det}
|\det \mathcal{A}| = |\det Y|.
\end{equation}
\end{lem}

\begin{proof} 
This in turn follows from a sequence of subclaims:
\medskip
\begin{itemize}
\item Let $b_1, \dots ,b_{d-1}$ be column vectors in $\mathbb{R}^d$, and let $\mathcal{B}$ be the matrix obtained from $\mathcal{A}$ by replacing $a_{11}, \dots , a_{1(d-1)}$ with $b_1, \dots ,b_{d-1}$. Suppose that $b_1 \wedge \dots \wedge b_{d-1} = a_{11} \wedge \dots \wedge a_{1(d-1)}$. Then we assert that
$$|\det\mathcal{B}| = |\det\mathcal{A}|.$$ 
(Recall that the exterior product $b_1 \wedge \dots \wedge b_{d-1}$ is the column vector in $\mathbb{R}^d$ given by determinant of the matrix whose first column consists of the standard unit vectors and whose subsequent columns are $b_1, \dots ,b_{d-1}$. It has direction given by the normal to the hyperplane spanned by $b_1, \dots ,b_{d-1}$ and magnitude the $(d-1)$-dimensional volume of the paralleotope generated by $b_1, \dots b_{d-1}$).

\medskip
\noindent
To see this, we may assume that both $\{a_{11}, \dots ,a_{1 (d-1)}\}$ and $\{b_1, \dots ,b_{d-1}\}$ are linearly independent sets in $\mathbb{R}^d$. Suppose that $b_1 \wedge \dots \wedge b_{d-1} = a_{1} \wedge \dots \wedge a_{d-1}$, where we have  temporarily re-labelled the $a_{1j}$ by dropping the first subscript $1$. Then we have that ${\rm span}\{b_1, \dots ,b_{d-1}\} = {\rm span}\{a_1, \dots ,a_{d-1}\} :=V$ (since the normal directions coincide) and that there is a linear transformation $B$ of $V$ of unit determinant taking the basis 
$\{a_1, \dots ,a_{d-1}\}$ of $V$ to the basis $\{b_1, \dots ,b_{d-1}\}$. Now extend this linear transformation $B$ to $\tilde{B}$ defined on the column space of the columns of $\mathcal{A}$ by defining it to take a column $c$ of $\mathcal{A}$ to itself, except when $c$ belongs to the second block of columns. In this case $c$ consists of some  column $a_{1k}$ followed by $d^2 -d$ zeros, and $\tilde{B}c$ is defined to be $B{a_{1k}}$ followed by $d^2 -d$ zeros. This extended linear transformation $\tilde{B}$ has block diagonal form with $B$ and $I_{d^2-d+1}$ forming diagonal blocks, and thus has determinant $1$. By construction we have $\mathcal{B} = \tilde{B}\mathcal{A}$, and so the subclaim follows from multiplicativity of determinants.

\medskip
\item There exists a function $F:\mathbb{R}^d \times \dots \times \mathbb{R}^d \to \mathbb{R}$ such that
$$ \det\mathcal{A} = F(y_1, \dots, y_d).$$
Indeed, by the previous subclaim, for $y_1, \dots , y_d$ fixed, the value of $\det\mathcal{A}$ is independent of the vectors $a_{jk}$ so long as they satisfy satisfy $y_j = a_{j1} \wedge \dots \wedge a_{j(d-1)}$ for each $j$. 

\medskip
\item
The function $F$ is multilinear: multiplication by scalars is immediate; now let us see additivity in the first variable. We may suppose that $y_1$ and $y_1'$ are not parallel. We need to see that 
\begin{equation}\label{eq:lin}
F(y_1 + y_1', y_2, \dots, y_d) = F(y_1, y_2, \dots, y_d) + F(y_1', y_2, \dots, y_d).
\end{equation}
Now the normal spaces $V$ and $V'$ to $y_1$ and $y_1'$ are of dimension $d-1$, and their intersection $V \cap V'$ is thus of dimension $d-2$. Pick any basis $\{a_2, \dots , a_{d-1}\}$ of $V \cap V'$ and extend it to bases $\{a_1, \dots , a_{d-1}\}$ and $\{a_1', \dots , a_{d-1}\}$ of $V$ and $V'$ respectively with the property that $y_1 = a_1 \wedge \dots \wedge a_{d-1}$ and 
$y_1' = a_1' \wedge \dots \wedge a_{d-1}$, and thus $y_1 + y_1' = (a_1 + a_1') \wedge \dots \wedge a_{d-1}$. Use these vectors as representatives to calculate the three terms appearing in \eqref{eq:lin}; then linearity of the determinant of $\mathcal{A}$ in the argument corresponding to $a_1$ gives the desired identity.

\medskip
\item
The function $F$ is alternating: if $\mathcal{A}$ has two identical nontrivial $d \times (d-1)$ blocks, then $\det\mathcal{A} = 0$. To see this, we may assume the two identical blocks $A$ occur in the first and second rows of blocks. It suffices to show that the row span of the first two blocks is strictly less than $2d$. And since from the fourth column block onwards all the entries in both row blocks are zero we may ignore them. It therefore suffices to show that the $2d \times (3d-2)$ block matrix
$$\begin{pmatrix}
           -I &  A  &  0  \\
           -I &  0  &  A \\   
\end{pmatrix}$$
has row rank strictly less than 2d. Block row operations reduce this matrix to 
$$\begin{pmatrix}
           -I &  A  &  0  \\
            0 &  -A  &  A \\   
\end{pmatrix}$$
so it suffices to show that the second row block here -- i.e. $( - A \;\; A)$ -- has row rank strictly less than $d$. But the columns of $( - A \;\; A)$ span a space of dimension at most $d-1$, and thus its column rank, and hence its row rank, are strictly less than $d$.

\medskip
\item
$F(e_1, \dots, e_d) = \pm 1$ -- this is verified by direct calculation. Indeed, with row and column operations, we may reduce the matrix $\mathcal{A}$ to $\pm I$.

\medskip
\item
Thus by the axiomatic characterisation of determinants, $F$ is actually the determinant function up to sign. 
\end{itemize}

\medskip
This finishes the proof of the lemma. 
\end{proof}

\textbf{Acknowledgements.} The first  author acknowledges support by the Hellenic Foundation for Research and Innovation (H.F.R.I.) under the call “Basic research Financing (Horizontal support of all Sciences)” under the National Recovery and Resilience Plan “Greece 2.0” funded by the European Union–NextGeneration EU(H.F.R.I. Project Number:15445).

\bibliographystyle{plain}

\bibliography{geometric_quantities}
\end{document}